\documentclass[12pt]{amsart}
\usepackage{amsmath,amsfonts,tikz, amsthm}
\usepackage{fullpage} % savetrees had margins too exotic for my printer ;)

\title{Set partitions with no $m$-nesting}

\author{Marni Mishna}
\address{Marni Mishna, Dept. Mathematics, Simon Fraser University,
  Burnaby, BC, Canada}
\author{Lily Yen}
\address{Lily Yen, Dept. Mathematics and Statistics, Capilano University, North
  Vancouver, and Dept. Mathematics, Simon Fraser University, Burnaby, BC, Canada}

\tikzset{partition/.style={fill,circle,inner sep=1pt}}
\newtheorem{prop}{Proposition}
\theoremstyle{definition}
\newtheorem{dfn}{Definition}
\theoremstyle{remark}

\newtheorem{example}{Example}
\newcommand{\nest}{\operatorname{ne}}

\begin{document}

\begin{abstract} % Abstract should precede \maketitle in amsart
  A partition of $\{1, \dots, n\}$ has an $m$-nesting if there exists
  $i_1<i_2< \dots < i_m <j_m < j_{m-1} < \dots < j_1$, where $i_l$ and
  $j_l$ are in the same block for all $1\le l \le m$. We use
  generating trees to construct the class of partitions with no
  $m$-nesting and determine functional equations satisfied by the
  associated generating functions.

  We use algebraic kernel method together with a linear operator to
  describe a coefficient extraction process. This gives rise to
  enumerative data, and illustrates the increasing complexity of the
  coefficient formulas as $m$ increases.
\end{abstract}

\maketitle

\section{Introduction}
In this work we address the enumeration of set partitions that avoid a
particular class of patterns. The patterns considered here, known as
$m$-nestings, arise in a graphical representation of set
partitions. Our goal is to determine useful enumerative information
about partitions that contain no $m$-nesting. This work is in the context
of recent studies of other combinatorial objects that avoid similar or
related patterns, in particular in the study of protein
folding~\cite{chr}.  Our strategy parallels a recent generating tree
approach used by Bousquet-M\'elou to enumerate a class of pattern
avoiding permutations~\cite{b-m}. Here, in a less condensed form, upon exhibiting a general
functional equation for all $m$, we apply both an appropriate
transformation of variables and a multiplicative factor to produce
symmetry in the kernel of the functional equation. This permits us  to apply
the algebraic kernel method and generate a telescoping sum, which
greatly simplifies the expression. The enumerative formulas are obtained by coefficient extraction.
 The result of the analysis produces a first
look into how such numbers are composed in a recurrence. It is our hope that readers 
new to the generating tree approach for obtaining  multivariate functional equations could understand the enumerative power and closed form limitations of this method.

\subsection{Notation and definitions}
A set partition $\pi$ of $[n] \mathrel{\mathop:}= \{ 1, 2, 3, \dots,
n\}$, denoted by $\pi \in \Pi_n$, is a collection of nonempty and
mutually disjoint subsets of $[n]$, called \emph{blocks}, whose union
is $[n]$. The number of set partitions of~$[n]$ into~$k$ blocks is
denoted~$S(n,k)$, and is known as Stirling number of the second
kind. The total number of partitions of $[n]$ is the \emph{Bell}
number $B_n = \sum_{k} S(n,k)$. We represent $\pi$ by a graph on the
vertex set $[n]$ whose edge set consists of arcs connecting elements
of each block in numerical order. Such an edge set is called the
\emph{standard representation} of the partition
$\pi$, as seen in~\cite{cddsy}.
 For example, the standard representation of
$$1|2\,5\,6\,8|3\,7|4$$ is given by the following graph with edge set~$\{
(2,5), (5,6), (6,8), (3,7)\}$:
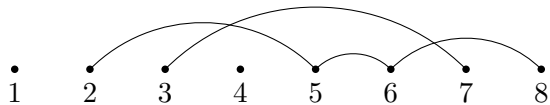
\begin{center}
\small
\begin{tikzpicture}
  \foreach \i in {1,...,8}
     \node[partition,label=below:$\i$] at (\i,0) {};
  \draw (2,0) to [bend left=45] (5,0);
  \draw (5,0) to [bend left=45] (6,0);
  \draw (6,0) to [bend left=45] (8,0);
  \draw (3,0) to [bend left=45] (7,0);	
\end{tikzpicture}
\end{center}

With this representation, we can define two classes of patterns:
crossings and nestings. An $m$-\emph{crossing} of~$\pi$ is a collection
of~$m$ edges~$(i_1, j_1)$, $(i_2, j_2)$, \dots, $(i_m, j_m)$ such that
$i_1<i_2< \dots < i_m <j_1 < j_2 < \dots < j_m$. Using the standard
representation, an $m$-crossing is drawn as follows:
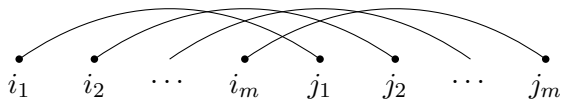
\begin{center}
\small
\begin{tikzpicture}[bend left=35]
  \node[partition,label=below:$i_1$] at (1,0) {};
  \node[partition,label=below:$i_2$] at (2,0) {};
  \node[label=below:$\dots$] at (3,0) {};
  \node[partition,label=below:$i_m$] at (4,0) {};
  \node[partition,label=below:$j_1$] at (5,0) {};
  \node[partition,label=below:$j_2$] at (6,0) {};
  \node[label=below:$\dots$] at (7,0) {};
  \node[partition,label=below:$j_m$] at (8,0) {};
  \draw (1,0) to (5,0);
  \draw (2,0) to (6,0);
  \draw (3,0) to (7,0);
  \draw (4,0) to (8,0);	
\end{tikzpicture}
\end{center}

Similarly, we define an $m$-\emph{nesting} of $\pi$ to be a collection
of $m$ edges $(i_1, j_1)$, $(i_2, j_2)$, \dots, $(i_m, j_m)$ such that
$i_1<i_2< \dots < i_m <j_m < j_{m-1} < \dots < j_1$. This is drawn:
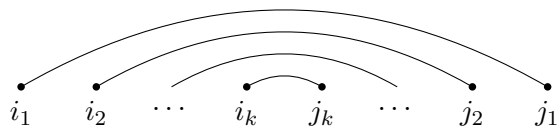
\begin{center}
\small
\begin{tikzpicture}[bend left=30]
  \node[partition,label=below:$i_1$] at (1,0) {};
  \node[partition,label=below:$i_2$] at (2,0) {};
  \node[label=below:$\dots$] at (3,0) {};
  \node[partition,label=below:$i_k$] at (4,0) {};
  \node[partition,label=below:$j_k$] at (5,0) {};
  \node[label=below:$\dots$] at (6,0) {};
  \node[partition,label=below:$j_2$] at (7,0) {};
  \node[partition,label=below:$j_1$] at (8,0) {};
  \draw (1,0) to (8,0);
  \draw (2,0) to (7,0);
  \draw (3,0) to (6,0);
  \draw (4,0) to (5,0);	
\end{tikzpicture}
\end{center}

A partition is $m$-noncrossing if it contains no $m$-crossing, and it
is said to be $m$-nonnesting if it contains no $m$-nesting.

%%Need to say something about enhanced m-crossings (nestings)
%%if we finish treating enhanced nestings also.

\subsection{Context and plan}
Chen, Deng, Du, Stanley and Yan in~\cite{cddsy} and Krattenthaler in~\cite{Kratt06} gave a non-trivial
bijective proof that $m$-noncrossing partitions of $[n]$ are
equinumerous with $m$-nonnesting partitions of $[n]$, for all values
of $m$ and $n$. A straightforward bijection with Dyck paths
illustrates that $2$-noncrossing partitions (also called
noncrossing partitions) are counted by Catalan
numbers. Bousquet-M\'elou and Xin in \cite{b-mx} showed that the
sequence counting $3$-noncrossing partitions is P-recursive, that is,
satisfies a linear recurrence relation with polynomial
coefficients. Indeed, they determined an explicit recursion, complete
with solution and asymptotic analysis. They
further conjectured that $m$-noncrossing partitions are not
P-recursive for all $m\ge 4$. Bell numbers are well known not to be
P-recursive because of the composed exponentials in the generating
function $B(x) = e^{e^x -1}$
as explained in Example $19$ of ~\cite{bb-mdfgg-b}.

Since $m$-noncrossing partitions of~$[n]$ and $m$-nonnesting
partitions of $[n]$ are equinumerous, we study $m$-nonnesting partitions in this
paper and show how to generate the class using generating trees, and
how to determine a recursion satisfied by the counting sequence for
$m$-nonnesting partitions.

Our approach is heavily inspired by Bousquet-M\'elou's recent work on the
enumeration of permutations with no long monotone subsequence
in~\cite{b-m}. She combined the ideas of recursive construction for
permutations via generating trees and the algebraic kernel method to
determine and solve functional equations with multiple catalytic variables.

In Section~\ref{sec:gt}, we employ Bousquet-M\'elou's generating
tree construction to find functional equations satisfied by the
generating functions for set partitions with no $m$-nesting. The
resulting equations, though similar to the equations arising
in~\cite{b-m}, need a similar multiplicative factor but a different
transformation of variables before a comparable analysis using
algebraic kernel method techniques is applied. To succeed in obtaining
information after applying the algebraic kernel method, a coefficient
extraction procedure is required. 
 This is completed in Section~\ref{sec:eqn-v}. Unfortunately for us, unlike her
work, we can only find explicit equations parameterised by $m$, with $m$
catalytic variables for all $m \ge 1$ without obtaining a closed form expression for the number of set partitions of size $n$ avoiding an $m$-nesting for $m \ge 3$.  In the case of $m=3$, however,  it is
similar to the functional equation given by Bousquet-M\'elou and Xin.

We are able to provide new enumerative data for $m>4$, and also offer
evidence supporting the non-P-recursive conjecture of Bousquet-M\'elou
and Xin in Section~\ref{sec:complexity}.

%--------------------------------------------------
\section{Generating Trees and Functional Equations}
\label{sec:gt}
%--------------------------------------------------
The generating tree construction for the class of $m$-nonnesting
partitions is based on a standard generating tree description of
partitions, and the constraint is incorporated using a vector
labelling system. The generating tree construction has an immediate
translation to a functional equation with $m$-variate series.

%--------------------------------------------------
\subsection{A generating tree for set partitions}
\label{sec:gt-sp}
%--------------------------------------------------

Let $\pi$ be a set partition. Define $\nest(\pi)$ to be the maximal $i$ such that $\pi$ has an $i$-nesting, also called the \emph{maximal nesting number}
 of $\pi$, and let $\Pi^{(m)}_n$ be the set of partitions
of $[n]$ for $n \ge 0$ (where $n=0$ means the empty partition) with
$\nest(\pi) \le m$, thus $(m+1)$-nonnesting.

Note that an arc over a fixed point is not a $2$-nesting, but a $1$-nesting.
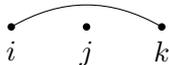
\begin{center}
\small
\begin{tikzpicture}[bend left=30]
  \node[partition, label=below:$i$] at (1,0) {};
  \node[partition, label=below:$k$] at (3,0) {};
  \node[partition,label=below:$j$] at (2,0) {};
  \draw (1,0) to (3,0);
\end{tikzpicture}
\end{center}

We next describe how to generate all set partitions via generating
trees in the fashion of~\cite{bb-mdfgg-b}. First, order the blocks of
a given partition, $\pi$, by the maximal element of each block in
descending order.
\begin{example}
  The first block of $1|2\,5\,6\,8|3\,7|4$ is $2\,5\,6\,8$; the second
  block is $3\,7$; the third block is singleton $4$; and $1$ is the
  last block. Using the standard representation,
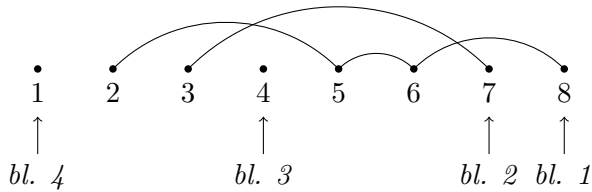
\begin{center}
\small
\begin{tikzpicture}[pin edge={<-,thin,black}]
  \foreach \i in {1,...,8}
     \node[partition,label=below:$\i$] at (\i,0) {};
  \draw (2,0) to [bend left=45] (5,0);
  \draw (5,0) to [bend left=45] (6,0);
  \draw (6,0) to [bend left=45] (8,0);
  \draw (3,0) to [bend left=45] (7,0);
  \node[pin=-90:\emph{bl. 4}] at (1,-.5) {};
  \node[pin=-90:\emph{bl. 3}] at (4,-.5) {};
  \node[pin=-90:\emph{bl. 2}] at (7,-.5) {};
  \node[pin=-90:\emph{bl. 1}] at (8,-.5) {};
\end{tikzpicture}
\end{center}
we number the blocks in descending order (from the right to the left)
according to the maximal element in each block (that is, the
rightmost vertex of each block).
\end{example}

With the order of blocks thus defined, we warm up  by generating all set partitions
without nesting restriction first. Figure~\ref{fig:gtrees-children}
contains the generating tree for all set partitions, in addition to the
generating tree for the number of children of each node from the tree
of set partitions to indicate how enumeration can be facilitated.

\begin{enumerate}
\item Begin with $\emptyset$ as the top node of the tree. It has only one child, so the corresponding node in the tree for the number of children is labelled $1$.

\item To produce the $n+1$st level of nodes, take each set partition at the $n$th level,  and either add $n+1$ as a singleton, or join $n+1$ to block $j$ for each $ 1\le j \le k$ if the set partition has $k$ blocks.
\end{enumerate}

Summarizing the description above in the notation of~\cite{bb-mdfgg-b}, we recall that  the rewriting rule of a generating tree is denoted by:
\[
[(s_0), \{(k) \rightarrow (e_{1,k}) (e_{2,k}) \dots (e_{k,k})\} ],
\]
where $s_0$ denotes the degree of the root, and for any node labelled $k$, that is, with $k$ descendents,  the label of each  descendent is given by $(e_{j,k})$ for $1 \le j \le k$. Thus, the class of set
partitions has a generating tree of labels given by $[(1): (k)\rightarrow(k+1)(k)^{k-1}].$

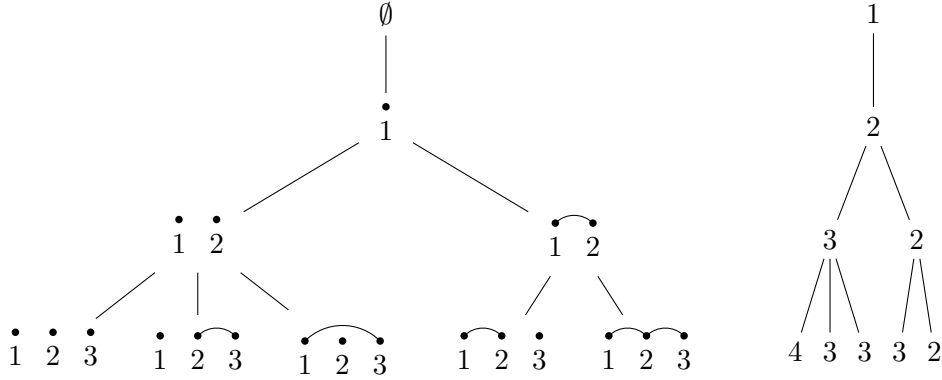
\begin{figure}[ht]  % BEGIN FIGURE %%%%%%%%%%%%%%%%%%%%%%%%%%%%%%%%%%%%%%%%%%
\small\centering
\newsavebox{\Pone}%
\newsavebox{\Ptwoone}%
\newsavebox{\Ptwotwo}%
\newsavebox{\Pthreeone}%
\newsavebox{\Pthreetwo}%
\newsavebox{\Pthreethree}%
\newsavebox{\Pthreefour}%
\newsavebox{\Pthreefive}%
\sbox{\Pone}{%
\begin{tikzpicture}[scale=0.5]
  \foreach \i in {1}
     \node[partition,label=below:$\i$] at (\i,0) {};
\end{tikzpicture}}%
\sbox{\Ptwoone}{%
\begin{tikzpicture}[scale=0.5]
  \foreach \i in {1,2}
     \node[partition,label=below:$\i$] at (\i,0) {};
\end{tikzpicture}}%
\sbox{\Ptwotwo}{%
\begin{tikzpicture}[scale=0.5]
  \foreach \i in {1,2}
     \node[partition,label=below:$\i$] at (\i,0) {};
  \draw (1,0) to [bend left=45] (2,0);
\end{tikzpicture}}%
\sbox{\Pthreeone}{%
\begin{tikzpicture}[scale=0.5]
  \foreach \i in {1,2,3}
     \node[partition,label=below:$\i$] at (\i,0) {};
\end{tikzpicture}}%
\sbox{\Pthreetwo}{%
\begin{tikzpicture}[scale=0.5]
  \foreach \i in {1,2,3}
     \node[partition,label=below:$\i$] at (\i,0) {};
  \draw (2,0) to [bend left=45] (3,0);
\end{tikzpicture}}%
\sbox{\Pthreethree}{%
\begin{tikzpicture}[scale=0.5]
  \foreach \i in {1,2,3}
     \node[partition,label=below:$\i$] at (\i,0) {};
  \draw (1,0) to [bend left=45] (3,0);
\end{tikzpicture}}%
\sbox{\Pthreefour}{%
\begin{tikzpicture}[scale=0.5]
  \foreach \i in {1,2,3}
     \node[partition,label=below:$\i$] at (\i,0) {};
  \draw (1,0) to [bend left=45] (2,0);
\end{tikzpicture}}%
\sbox{\Pthreefive}{%
\begin{tikzpicture}[scale=0.5]
  \foreach \i in {1,2,3}
     \node[partition,label=below:$\i$] at (\i,0) {};
  \draw (1,0) to [bend left=45]
        (2,0) to [bend left=45] (3,0);
\end{tikzpicture}}%
\begin{tikzpicture}[
     level 2/.style={sibling distance=13em},
     level 3/.style={sibling distance=5em},
     baseline=(current bounding box.base)
  ]
  \node {$\emptyset$}
     child { node {\usebox{\Pone}}
        child { node {\usebox{\Ptwoone}}
           child { node {\usebox{\Pthreeone}} }
           child { node {\usebox{\Pthreetwo}} }
           child { node {\usebox{\Pthreethree}} }
        }
        child { node {\usebox{\Ptwotwo}}
           child { node {\usebox{\Pthreefour}} }
           child { node {\usebox{\Pthreefive}} }
        }
     };
\end{tikzpicture}
\qquad
\begin{tikzpicture}[
     level 2/.style={sibling distance=3em},
     level 3/.style={sibling distance=1.2em},
     baseline=(current bounding box.base)
  ]
  \node {1}
     child { node {2}
        child { node {3}
           child { node {4} }
           child { node {3} }
           child { node {3} }
        }
        child { node {2}
           child { node {3} }
           child { node {2} }
        }
     };
\end{tikzpicture}
\caption{\emph{Generating tree for set partitions and its corresponding generating tree of the number of children}}
\label{fig:gtrees-children}
\end{figure}   % END FIGURE %%%%%%%%%%%%%%%%%%%%%%%%%%%%%%%%%%

\subsection{A vector label to track nestings}
Note that in Figure~\ref{fig:gtrees-children}, the generating tree of
set partitions generates all set partitions graded by $n$, the size of
$\pi$ but it does not keep track of nesting numbers. Also note that in
the generating tree for the number of children, the number of children
of $\pi$ is one more than the number of blocks of $\pi$ for any
partition $\pi$.

Fix $m$. In order to keep track of nesting numbers, we need to define the
\emph{label} of $\pi \in \Pi^{(m)}$.

\begin{dfn}
Define the label of a partition, $L(\pi) = ( a_1(\pi), a_2(\pi), \dots, a_m(\pi))$, or in short, $L(\pi) = ( a_1, a_2, \dots, a_m)$ as follows. For $1 \le j \le m$,
\[
  a_j(\pi) =
  \begin{cases}
     \text{\parbox{3.2in}{$1+$ number of blocks in $\pi$,
                        }}&  \text{if $\pi$ is $j$-nonnesting,}\\
&\\
     \text{\parbox{3.2in}{$1+$ number of blocks ending to the right
                        of the smallest vertex in the
                        rightmost~$j$-nesting}}
         & \text{otherwise.}
  \end{cases}
\]

\end{dfn}
By the \emph{rightmost}~$j$-nesting, we mean the minimal element in the $j$-nesting of a particular partition $\pi$ that is greater than or equal to all minimal elements in all $j$-nestings of $\pi$.

\begin{example}
To continue the example, let $\pi=1|2\,5\,6\,8|3\,7|4$ and suppose
$m=3$. Then $L(1|2\,5\,6\,8|3\,7|4) = (3,4,5)$ for the following
reasons. The rightmost 1-nesting is the edge with largest vertex
endpoint, $(6,8)$. Hence, $a_1(\pi)=3$ because blocks 1 and 2 end
to the right of vertex 6. The rightmost 2-nesting is the set of edges
$\{(5,6), (3,7)\}$ hence $a_2(\pi)=4$ because 3 blocks end to the
right of vertex 3. Finally,
$a_3(\pi)=5$ because the diagram has no $3$-nesting, and is comprised of
$4$ blocks.
Note that in this convention, the empty set partition has label
$(1,1,\dots, 1)$, since it has no nestings and no blocks.
\end{example}

A set partition in~$\Pi^{(m)}$ will have $a_m$ children. This is one plus the
number of blocks, if there is no $m$-nestings (and hence no risk that
adding an edge will create an $m+1$-nesting). Otherwise, it indicates one plus
the number of blocks to which you can add an edge without creating an
$m+1$-nesting. The label of a set partition is sufficient to derive
the label of each of its children, and this is described in the next
proposition. Also, remark that the label is a non-decreasing sequence, since
the rightmost $j$-nesting either contains the rightmost $j-1$ nestings
or is to the left of it.

\begin{prop}[Labels of children]
\label{prop:clabel}
Let $\pi$ be in $ \Pi^{(m)}_n$, the set of set partitions on $[n]$
avoiding $m+1$-nestings, and suppose the label of $\pi$ is $L(\pi) = (
a_1, a_2, \dots, a_m)$. Then, the labels of the $a_m$ set partitions
of $\Pi^{(m)}_{n+1}$ obtained by recursive construction via the
generating tree are
\[
(a_1+1, a_2+1, \dots, a_m+1) \qquad\mbox{(Add $n+1$ as a singleton to $\pi$)}
\]

and
\[
\begin{array}{l@{\quad(}*{3}{r@{,\,}}@{\dots,\,}r@{)\qquad}l}
     &  2    & a_2  & a_3   & a_{m-1}, a_m&\mbox{(Add $n+1$ to
    block 1)}\\
      &  3    & a_2  & a_3    & a_{m-1}, a_m&\mbox{(Add $n+1$ to
    block 2)}\\
  \multicolumn{5}{r}{}&\multicolumn{1}{c}{\vdots}\\
 &  a_1& a_2& a_3    & a_{m-1}, a_m&\mbox{(Add $n+1$ to
    block $a_1-1$)}\\
   &  a_1+1& a_1+1& a_3  & a_{m-1}, a_m&\mbox{(Add $n+1$ to
    block $a_1$)}\\
 &  a_1+1& a_1+2& a_3 & a_{m-1}, a_m& \mbox{(Add $n+1$ to
    block $a_1+1$)}\\
  \multicolumn{5}{r}{}&\multicolumn{1}{c}{\vdots}\\
   &  a_1+1& a_2+1& a_2+1  & a_{m-1}, a_m& \mbox{(Add $n+1$ to
    block $a_2$)}\\
 \multicolumn{5}{r}{}&\multicolumn{1}{c}{\vdots}\\
  &  a_1+1& a_2+1& a_3+1& a_{m-1}+1, a_{m-1}+1&\mbox{(Add $n+1$ to
    block $a_{m-1}$)}\\
\multicolumn{5}{r}{}&\multicolumn{1}{c}{\vdots}\\
&  a_1+1& a_2+1& a_3+1& a_{m-1}+1, a_m&\mbox{(Add $n+1$ to
    block $a_{m} - 1$)}\\
\end{array}
\]
\end{prop}

\begin{proof}
By careful inspection.
\end{proof}

\begin{example}
Consider the following partition from $\Pi_{8}^{(3)}$. The reader can refer to its arc diagram in Example 1 which shows that is is $3$-nonnesting, thus also $4$-nonnesting.
The partition $1|2\,5\,6\,8|3\,7|4$ with label
$(3,4,5)$ has five children and their respective labels are:
\[\begin{array}{cc}
\pi & L(\pi)\\[1mm]
1|2\,5\,6\,8|3\,7|4|9 & (4, 5, 6)\\
1|2\,5\,6\,8\,9|3\,7|4& (2, 4, 5)\\
1|2\,5\,6\,8|3\,7\,9|4& (3, 4, 5)\\
1|2\,5\,6\,8|3\,7|4\,9& (4, 4, 5)\\
1\,9|2\,5\,6\,8|3\,7|4& (4, 5, 5)\\
\end{array}
\]
\end{example}

Notice that in Proposition~\ref{prop:clabel}, the first label comes
from adding $n+1$ as a singleton to get an element of
$\Pi^{(m)}_{n+1}$. The other $a_m - 1$ labels result from adjoining
the element $n+1$ to the maximal element of block $l$ for every $1 \le
l \le a_m -1$ blocks without creating an $m+1$-nesting.
\begin{example}As we mentioned before, $2$-nonnesting set partitions are counted
by Catalan numbers. The generating tree construction given in
Proposition~\ref{prop:clabel} restricted to this case is given by \[[(1):
(k)\rightarrow (2)(3)\dots(k+1)],\] which is the same construction for
Catalan numbers given in~\cite{bb-mdfgg-b}.
The generating tree for 3-noncrossing partitions is given by
\[
[(1,1):(i,j)\rightarrow(i+1, j+1)(2,j)(3,j)\cdots(i,j)(i+1,
i+1)(i+1, i+2)\dots(i+1, j)].
\]
\end{example}
% Maple code for generating the above generating tree.
% rule:=proc(label) option remember;
%  local h,r,s,out;
%  h:=label[1]; r:=label[2];;
%  out:=[[h+1,r+1],seq([i, r],i=2..h),seq([h+1, j],j=h+1..r) ];
%  return out;
% end:
% nextlevel:=proc(l) option remember;
%  local i,out;
%  out:=[];
%  for i from 1 to nops(l) do
%   out:=[op(out),op(rule(l[i]))];
%  od;
%  return out;
% end:
% level:=proc(n) option remember;
%   if n=0 then
%      return [[1,1]]
%   else
%      return nextlevel(level(n-1));
%   fi;
% end:
% seq(nops(level(n)),n=0..7);
%                         1, 1, 2, 5, 15, 52, 202, 859

%----------------------------------------------------------------------
\subsection{A functional equation for the generating function}
\label{sec:u-eqs}
%----------------------------------------------------------------------
The simple structure of the labels of a partition's children in
Proposition \ref{prop:clabel} permits a straightforward translation of
the combinatorial construction into a functional equation. Let us
define $\tilde{F}(u_1, u_2, \dots, u_m; t)$ to be the ordinary
generating function of partitions in $\Pi^{(m)}$ counted by the
statistics $a_1$, $a_2$, \dots, $a_m$ and by size,
\[
\begin{split}
\tilde{F}(u_1, u_2, \dots, u_m; t)  &:= \sum_{\pi \in \Pi^{(m)}} u_1^{a_1(\pi)}u_2^{a_2(\pi)} \dots u_m^{a_m(\pi)}t^{|\pi|}\\
				    &= \sum_{a_1, a_2, \dots, a_m} \tilde{F}_{\mathbf{a}}(t) u_1^{a_1}u_2^{a_2} \dots u_m^{a_m},	
\end{split}
\]
where $\tilde{F}_{\mathbf{a}}(t)$ is the size generating
function for the set partitions of $\Pi^{(m)}$ with the label $\mathbf{a}=(a_1, a_2,
\dots, a_m)$.
Thus, when $m=2$,
\[
\tilde{F}(\mathbf{u};t) = u_{{1}}u_{{2}}+{u_{{1}}}^{2}{u_{{2}}}^{2}t+ \left( {u_{{1}}}^{3}{u_{{
2}}}^{3}+{u_{{1}}}^{2}{u_{{2}}}^{2} \right) {t}^{2}+ \left( {u_{{1}}}^
{4}{u_{{2}}}^{4}+2\,{u_{{1}}}^{3}{u_{{2}}}^{3}+{u_{{1}}}^{2}{u_{{2}}}^
{2}+{u_{{1}}}^{2}{u_{{2}}}^{3} \right) {t}^{3}+\dots.
\]
Proposition~\ref{prop:clabel} implies
\begin{align*}
\tilde{F}(u_1, u_2, \dots, u_m; t) &= u_1u_2\dots u_m + tu_1u_2\dots u_m\tilde{F}(u_1, u_2, \dots, u_m; t)
\\
&+ t\sum_{a_1, a_2, \dots, a_m}\hspace{-2mm}\tilde{F}_{\mathbf{a}}(t) u_2^{a_2}u_3^{a_3}\dots u_m^{a_m}\sum_{\alpha = 2}^{a_1} u_1^{\alpha}
\\
&+ t\sum_{a_1, a_2, \dots, a_m}\hspace{-2mm}\tilde{F}_{\mathbf{a}}(t)\sum_{j=2}^m \sum_{\alpha = a_{j-1}+1}^{a_j} u_1^{a_1 + 1}
u_2^{a_2 + 1}\dots u_{j-1}^{a_{j-1} + 1} u_j^\alpha u_{j+1}^{a_{j+1}} \dots u_m^{a_m}.
\end{align*}

We compactify the sums using finite geometric series sum formula and summarize the above derivation into the following functional equation for the generating function.
\begin{prop}

\begin{multline}
\tilde{F}(\mathbf{u}; t)
  = u_1u_2\dots u_m
     + tu_1u_2\dots u_m \tilde{F}(\mathbf{u}; t)
\\
  + tu_1 \left(\frac{ \tilde{F}(\mathbf{u};t) - u_1\tilde{F}(1,u_2, \dots, u_m; t) }
                                                {u_1 - 1} \right)
 \\
  + t\sum_{j=2}^m u_1u_2\dots u_j
       \left( \frac{\tilde{F}(\mathbf{u};t) -\tilde{F}(u_1,\dots,u_{j-2},u_{j-1}u_j,1, u_{j+1},\dots, u_m; t)}
             {u_j-1}\right),
\label{eq:u}
\end{multline}
where $\tilde{F}(\mathbf{u};t) = \tilde{F}(u_1, u_2, \dots, u_m; t)$.
\end{prop}

%-----------------------------------------------------------------
\section{Processing the functional equation}
\label{sec:eqn-v}
%-----------------------------------------------------------------
To process the functional equation for $\tilde{F}(\mathbf{u};t)$ we
transform the variables to get a form more amenable to analysis. We
follow~\cite{b-m}, and do this in two steps. The first rewrites in a
$v$ variable to remove the exponent restriction on the $u_i$'s because $a_1(\pi) \le a_2(\pi) \le \dots \le a_m(\pi)$; the second
transformation to a set of $x$
variables in the next section allows us to analyse the coefficients.

\subsection{Removing the exponent restriction}
Define
\[
F(\mathbf{v};t)=F(v_1, \dots, v_m, v_{m+1};t):=
\sum_{\pi \in \Pi^{(m)}} v_1^{a_1}v_2^{a_2-a_1} \dots v_m^{a_m-a_{m-1}} v_{m+1}^{|\pi|-a_m} t^{|\pi|}.
\]
where $(a_1, a_2, \dots, a_m) = L(\pi)$. Thus, we have eliminated the dependency $a_1\le a_2 \le \dots \le a_m$ between the exponents of $u_1$, $u_2$, \dots, $u_m$ in $\tilde{F}(\mathbf{u};t)$.

We can write $\tilde{F}$ in terms of $F$ and vice versa:
\[
F(v_1, \dots, v_{m+1}; t) =\tilde{F}\left(\frac{v_1}{v_2}, \frac{v_2}{v_3}, \dots, \frac{v_m}{v_{m+1}} ; v_{m+1}t\right),
\]
and
\[
\tilde{F}(u_1, \dots, u_m; v_{m+1}t) = F(u_{1,m}v_{m+1}, u_{2,m}v_{m+1}, \dots, u_{m}v_{m+1}, v_{m+1} ; t),
\]
where $u_{j,m} = u_j u_{j+1} \dots u_m$. The function $F$ satisfies a simpler functional equation.

\begin{prop}\label{veq}
The generating function $F(\mathbf{v};t) = F(v_1, v_2, \dots, v_{m}, v_{m+1};t)$ of set partitions of $\Pi^{(m)}$ satisfies
\begin{multline}
F(\mathbf{v};t) = \frac{v_1}{v_{m+1}} + tv_1 \Bigl(F(\mathbf{v};t)+ v_{m+1} \sum_{j=1}^m \frac{F(\mathbf{v};t)}{v_j - v_{j+1}}\Bigr)
\\
 - v_{m+1}tv_1 \Bigl(\frac{v_1}{v_2} \frac{F(v_2, v_2, v_3, \dots, v_{m+1};t)}{v_1 - v_2}
\\
  +  \sum_{j=2}^m \frac{F(v_1, \dots, v_{j-1}, v_{j+1}, v_{j+1}, v_{j+2}, \dots, v_m, v_{m+1};t)}
                                            {v_j - v_{j+1}}\Bigr).
\end{multline}
The series $F(1, 1, \dots, 1;t)$ is the generating function for the
class of  $(m+1)$-nonnesting set partitions.
\end{prop}

%----------------------------------------------------------------------
\subsection{A second transformation}
\label{sec:x-eqs}
%----------------------------------------------------------------------
We take the functional equation for $F(\mathbf{v};t)$ in Proposition
\ref{veq} and rearrange the terms to find the kernel of the functional equation as follows
\begin{multline}
\left(1-tv_1 - tv_1v_{m+1} \sum_{j=1}^m \frac1{v_j - v_{j+1}}\right) F(\mathbf{v} ;t) =\\
\frac{v_1}{v_{m+1}} - v_{m+1}tv_1
        \Bigl(\frac{v_1}{v_2} \frac{F(v_2, v_2, v_3, \dots, v_{m+1};t)}{v_1 - v_2}
   \\
   +    \sum_{j=2}^m \frac{F(v_1, \dots, v_{j-1}, v_{j+1}, v_{j+1}, v_{j+2}, \dots, v_m, v_{m+1};t)}
                                                   {v_j - v_{j+1}}
             \Bigr).
\end{multline}
The kernel is
\[
1-tv_1 - tv_1v_{m+1} \sum_{j=1}^m \frac1{v_j - v_{j+1}}.
\]
To exploit invariance properties of the kernel, we introduce the following transformation of the $v_j$'s:
\begin{align*}
v_{m+1} &= 1\\
v_{m}   &= 1 + x_m\\
\vdots\\
v_2     &= 1 + x_m + \dots + x_2\\
v_1     &= 1 + x_m + \dots + x_2 + x_1.
\end{align*}
This transformation enables us to rewrite the kernel as
\[
1 - t\left(x_1+x_2+\dots+x_m+1\right)\left(1+\sum_{j=1}^{m}\frac{1}{x_j}\right).
\]
This new kernel is invariant under~$\mathfrak{S}_m$, the symmetric group on $[m]$.
To simplify presentation of the functional equation, we use
\[
s=x_1+x_2+\dots+x_m+1, \qquad h = \frac1{x_1} + \frac1{x_2}+ \dots + \frac1{x_{m}} +1.
\]
Divide both sides of the functional equation by $s$ we just defined, we get
\begin{multline}
\bigl(\frac1{s} - t h\bigr) \bar{F}(x_1, x_2, \dots, x_m;t) =
\\
1 - t\Bigl( \frac{s}{s-x_1} \frac{\bar{F}(0, x_2, x_3, \dots, x_m;t)}{x_1}
\\
   +   \sum_{j=2}^m \frac{\bar{F}(x_1, \dots, x_{j-2}, x_{j-1} + x_j, 0, x_{j+1}, \dots, x_m; t)}
                                                    {x_j}\Bigr),
\label{eq:kerx}
\end{multline}
where
\[
\bar{F}(x_1, x_2, \dots, x_m;t) = F(v_1, v_2, \dots, v_m, v_{m+1};t),
\]
and
\[
\bar{F}(0, 0, \dots, 0;t) = F(1, 1, \dots, 1; t)
\]
is the evaluation that yields the ordinary generating function in $t$ for set partitions avoiding $m+1$-nestings.

%----------------------------------------------------------------------
\subsection{A multiplicative factor and a telescoping sum}
\label{sec:Mx}
%----------------------------------------------------------------------
We introduce a multiplicative factor, $M(\mathbf{x}) := x_1 x_2^2
x_3^3 \dots x_m^m$, to be applied to Equation~\eqref{eq:kerx}.  Let the
new kernel, $K(\mathbf{x};t)$ be defined by $ \frac1{s} - ht$. Since
the kernel $K(\mathbf{x};t)$ is invariant under $\mathfrak{S}_m$, when
we take the signed orbit sum of the functional Equation \eqref{eq:kerx}
under $\mathfrak{S}_m$, namely, $\sum_{\sigma \in \mathfrak{S}_m} \epsilon(\sigma) \sigma(\text{functional equation})$ the left hand side has the kernel as a factor
outside the sum; namely,
\[
LHS = K(\mathbf{x}; t)\sum_{\sigma \in \mathfrak{S}_m} \epsilon(\sigma) \sigma(M(\mathbf{x}) \bar{F}(x_1, \dots, x_m;t)).
\]

On the right hand side of Equation~\eqref{eq:kerx}, before taking the
orbit sum, the effect of multiplying by $M(\mathbf{x})$ is
\begin{equation}
\begin{split}
M(\mathbf{x}) - t \Bigl(&x_2^2 x_3^3 \dots x_m^m \bar{F}(0, x_2, x_3, \dots, x_m;t) +
\frac{M(\mathbf{x}) \bar{F}(0, x_2, \dots, x_m;t)}{x_2 + x_3+ \dots + x_m}
 \\
&+ x_1 x_2 x_3^3 \dots x_m^m \bar{F}(x_1+x_2, 0, x_3, \dots, x_m;t)
\\
&+ x_1 x_2^2 x_3^2 x_4^4 \dots x_m^m \bar{F}(x_1,x_2 + x_3, 0, x_4, \dots, x_m;t) \label{eq:mx}
\\
&+ \dots
 \\
&+ x_1 x_2^2 x_3^3 \dots x_{m-1}^{m-1} x_m^{m-1} \bar{F}(x_1,\dots, x_{m-2}, x_{m-1}+x_m, 0;t)\Bigr)
\end{split}
\end{equation}
Note that the coefficient of $\bar{F}(0, x_2, x_3, \dots, x_m;t)$ is split because
\[
\frac{s}{s-x_1} = \frac{s - x_1 + x_1}{s - x_1} = 1 + \frac{x_1}{s - x_1}
\]
which is easier to manipulate when the orbit sum is taken.
Because each of the last $m-1$ terms of the RHS of Equation~\eqref{eq:mx} is
invariant under $\sigma_j = (j, j+1)$ for some $j \in [m-1]$, (that is, the
generators for $\mathfrak{S}_m$), by forming the signed sum over
$\mathfrak{S}_m$ we reduce these $m-1$ terms to zero, leaving only the
first three terms:
\begin{multline}
K( \mathbf{x}; t ) \sum_{ \sigma \in \mathfrak{S}_m }
\epsilon( \sigma ) \sigma( M(\mathbf{x}) \bar{F}( x_1, \dots, x_m;t ) )\\
=\sum_{ \sigma \in \mathfrak{S}_m } \epsilon( \sigma ) \sigma( M( x ))
 -t\sum_{\sigma \in \mathfrak{S}_m} \epsilon(\sigma) \sigma( x_2^2 x_3^3 \dots x_m^m \bar{F}(0, x_2, \dots, x_m;t) ) \label{eq:orsum}\\
       -t\sum_{\sigma \in \mathfrak{S}_m}  \epsilon(\sigma)
\sigma \Bigl(\frac{ M(\mathbf{x}) \bar{F}(0, x_2, \dots, x_m;t)}{x_2+x_3+\dots + 1}
          \Bigr).
\end{multline}

%----------------------------------------------------------------------
\subsection{The constant term extraction operator $\mathcal{CT}$}
\label{sec:L}
%----------------------------------------------------------------------

Our goal is to obtain the series $\bar{F}(0, 0, \dots, 0;t)$. Remark,
any term in $\bar{F}(x_1, \dots, x_m;t)$ containing non-zero exponents
of $x_i$'s for $i \in [m]$ disappears when~$x_i$ is set to~$0$. The
exponents of each $(x_i + x_{i+1} + \dots + x_m+1)$ are all
non-negative, implying that to get a constant term, each factor in
parentheses must go to the constant, leaving only the variable $t$,
keeping track of the size of the partition. For the sake of brevity in presentation, we define a linear operator for constant term extraction, namely $[x_1^0 x_2^0 \dots x_m^0]\bar{F}$.
\begin{dfn}
\label{loL}
Let $\mathcal{CT}$ be the constant term extraction operator defined on Laurent series by the following action on monomials:
\[
\mathcal{CT}(x_1^{e_1} x_2^{e_2} \dots x_m^{e_m} t^k) =\begin{cases}
  0,&  \text{if $e_i \ne 0$ for some $i \in [m]$},\\
  t^k & \text{otherwise.}
  \end{cases}
\]
\end{dfn}
Before applying our constant term extraction operator $\mathcal{CT}$,  to the orbit sum, Equation~\eqref{eq:orsum}, we first divide Equation~\eqref{eq:orsum}
by $M(\mathbf{x})K(\mathbf{x};t)$:
\begin{multline}
\sum_{\sigma \in \mathfrak{S}_m} \frac{\epsilon(\sigma) \sigma(M(\mathbf{x}) \bar{F}(x_1, \dots, x_m;t))}{M(\mathbf{x})}\\
\shoveleft{=\frac{s}{1-ths}
\Bigl(\sum_{\sigma \in \mathfrak{S}_m} \frac{\epsilon(\sigma) \sigma(M(\mathbf{x}))}{M(\mathbf{x})}
                 } \\
- t\sum_{\sigma \in \mathfrak{S}_m} \frac{\epsilon(\sigma) \sigma(x_2^2x_3^3\dots x_m^m \bar{F}(0,x_2,\dots, x_m;t))}{M(\mathbf{x})} \\
- t\sum_{\sigma \in \mathfrak{S}_m}
\frac{ \epsilon(\sigma) \sigma(
                               \frac{ M(\mathbf{x}) \bar{F}(0, x_2,\dots, x_m;t) }
                                    { x_2+\dots+x_m+1}
                              )
     }{M(\mathbf{x})}
\Bigr).
\label{eq:morsum}
\end{multline}

On the LHS of Equation~\eqref{eq:morsum}, after $\mathcal{CT}$ is applied, only the term corresponding to $\sigma=id$ remains, yielding
\[
\mathcal{CT}(\bar{F}(x_1, \dots, x_m;t)) = \sum_{\pi \in \Pi^{(m)}}t^{|\pi|}
\]
because the other terms all contain a nonzero exponent for some $x_i$
where $i\in [m]$. When we extract the coefficient of $t^n$ from
$\mathcal{CT}(\bar{F}(\mathbf{x};t))$, we get precisely the number of
set partitions of size $n$ without an $(m+1)$-nesting.

The task is now clear: We need to extract the coefficient of
$x_1^0x_2^0\dots x_m^0$ from the RHS  of
Equation~\eqref{eq:morsum}. Fortunately, since $\mathcal{CT}$ is a
linear operator, we can examine the RHS of Equation~\eqref{eq:morsum}
term by term, namely, by considering the three surviving orbit sums
one at a time. We illustrate this process with an example using nonnesting set partitions.

%----------------------------------------------------------------------
\section{$2$-nonnesting set partitions}
%----------------------------------------------------------------------
The generating function derivation in this section is for pedagogical
purposes, to illustrate how to manipulate
Equation~\eqref{eq:morsum}. Indeed, there are easier ways to determine
the generating function for the Catalan numbers.

To enumerate $2$-nonnesting set partitions we set $m=1$ in the above
equations. In this case, $\bar{F}(0;t) = \sum_{\pi  \in
  \Pi^{(1)}} t^{|\pi|} $ which we rewrite as $\bar{F}(0;t) =\sum
F_n t^n$.  Since $m=1$, the associated symmetric group is
$\mathfrak{S}_1$ which only contains the identity permutation; thus
the functional equation is:
\begin{equation}
\bar{F}(x_1; t) = \frac{x_1 + 1}{ 1-t \left(\frac1{x_1} + 1\right)(x_1 + 1)}
                - \frac{t (x_1 + 1)} { 1-t \left(\frac1{x_1} + 1\right)(x_1 + 1)}\left(\frac1{x_1} + 1\right)\bar{ F}(0;t).
\label{eq:m1}
\end{equation}

Though this is an easy case, writing out the action of $\mathcal{CT}$ shows us how terms are collected and coefficients computed in a rather slow way. First expand Equation~\eqref{eq:m1} as power series to get
\begin{multline}
\bar{F}(x_1;t) =\sum_{n=0}^{\infty} (x_1+1)^{n+1} \left(\frac1{x_1}+1\right)^nt^n
     - \sum_{n=0}^{\infty}((x_1+1)t)^{n+1}\left(\frac1{x_1}+1\right)^{n+1}
       \times \sum_{n=0}^\infty F_n t^n
\end{multline}
Now apply $\mathcal{CT}$ to get
\begin{align*}
\sum_{n=0}^\infty F_n t^n =\mathcal{CT} \bar{F}(x_1;t)
&= \sum_n  \sum_{j = 0}^n \binom{n+1}{j+1} \binom{n}{j}t^n
 -\sum_n \sum_{j= 0}^n \binom{n+1}{j+1} ^2t^{n+1} \times
                            \sum_{n=0}^\infty F_n t^n
\\
&=\sum_n t^n\left(\sum_{j=0}^n \binom{n+1}{j+1} \binom{n}{j} - \sum_{n^* = 0}^{n-1}\Bigl ( \sum_{j=0}^{n^*+1} \binom{n^* + 1}{j}^2 \Bigr) F_{n-n^*-1}\right).
\end{align*}
We deduce a recurrence for $F_n$ after simplifying the binomial
summations:
\[
F_n= \frac{1}{2}\binom{2n+2}{n+1} -
\sum_{j=0}^{n-1} \binom{2j+2}{j+1}F_{n-1-j}.
\]
When all $F_k$'s are collected to the left, we get
\[
\sum_{k=0}^{n} \binom{2k}{k} F_{n-k} = \frac{1}{2}\binom{2n+2}{n+1}.
\]
Upon noticing that the left hand side is a convolution product, we define
\[
f(x) = \sum_{0}^{\infty} F_k x^k, \qquad \text{and} \qquad
g(x) = \sum_{0}^{\infty} \binom{2k}{k} x^k = \frac{1}{\sqrt{1-4x}}
\]
to obtain
\[
f(x) g(x) = \frac{1}{2x} (g(x) - 1), \qquad \text{or} \qquad 
f(x) = \frac{1}{2x}(1-\sqrt{1-4x}),
\]
the famous Catalan series as expected.
% P:= proc(n)
% if n<0 then 0
% elif n=0 then 1
% else
%   binomial(2*n+2, n+1)/2-add(P(n-1-k)*binomial(2*(k+1), k+1), k=0..n-1);
% fi
% end proc:
% seq(P(i), i=0..10);
%                1, 1, 2, 5, 14, 42, 132, 429, 1430, 4862, 16796

%----------------------------------------------------------------------
\section{$3$-nonnesting set partitions}
\label{sec:3non}
%----------------------------------------------------------------------
The first non-trivial case is $3$-nonnesting set partitions to study how
the orbit sum produces sums of products of multinomial
coefficients. Notice that the previous example with $m=2$, no explicit formula for $F_n$ was used; instead, it the $F_n$'s was defined in terms of all previous $F_j$ for all $j \le n$. However, the convolution product allowed a successful isolation of the generating function $f(x)$. For this reason, it is our opinion that the study of the structure of convolution-like product  for $m=3$ and beyond may shed light in the nature of the generating series. As in the previous example, through the investigating of  the action of $\mathcal{CT}$
on the RHS of Equation~\eqref{eq:morsum}, we show how the conditions
of summation indices turn out to reduce to a simple equation, thus
restricting the degree of freedom. This exercise, though tedious when
carried to the next case, $m = 3$, lends evidence to the conjecture by
Bousquet-M\'elou and Xin in \cite{b-mx} that the generating function
of the $4$-nonnesting case is not D-finite. Furthermore, we get
 enumerative formulas as functions of the label, and an understanding of the structure
of the generating functions.

%----------------------------------------------------------------------
\subsection{First term of RHS of Equation~\eqref{eq:morsum}}
%----------------------------------------------------------------------
The first extraction is resolved by a simple combinatorial
argument on the total way to combine the exponents to get a constant:
\begin{align*}
\mathcal{CT}&\Bigl( \frac{x_1 + x_2 + 1}
                      {1- t h s}    \sum_{\sigma \in \mathfrak{S}_2}
\epsilon(\sigma) \frac{\sigma(x_1 x_2^2)}
                     {x_1 x_2^2}
           \Bigr)\\
\qquad=&\mathcal{CT}\biggl( \Bigl( \sum_{n=0}^{\infty} \left(\frac1{x_1} + \frac1{x_2} + 1\right)^n \left( x_1 + x_2 + 1\right)^{n+1} t^n \Bigr)
                  (1 - \frac{x_1}{x_2})
            \biggr)\\
=&\sum_{n=0}^{\infty} \sum_{\substack{0 \le l_1, l_2, l_3 \\
                                    l_1 + l_2 + l_3 = n }   }
 \binom{n}
       {l_1, l_2, l_3} \binom{n+1}
                             {l_1, l_2, l_3 + 1} t^n \\
  &\qquad -  \sum_{n=0}^{\infty} \sum_{\substack{0 \le l_1, l_2, l_3 \\
                                        l_1 + l_2 + l_3 = n  }     }
 \binom{n}
       {l_1, l_2, l_3} \binom{n+1}
                             {l_1 - 1, l_2 + 1, l_3 + 1}  t^n\\
=&\sum_{n=0}^{\infty}t^n \sum_{\substack{0 \le l_1, l_2, l_3 \\
                                    l_1 + l_2 + l_3 = n }   }
 \binom{n}{l_1, l_2, l_3} \left(1-\frac{l_1}{l_2+1}\right).
\end{align*}

%----------------------------------------------------------------------
\subsection{Second term of RHS of Equation~\eqref{eq:morsum}}
%---------------------------------------------------------------------
The remaining terms are expressed in terms of
\[\bar{F}(0, x_2;t)=
\sum_{\pi \in \Pi^{(2)}} (x_2 + 1)^{a_2} t^{|\pi|}= \sum_{n=0}^\infty
\sum_{k=1}^{n+1} F_n(k) (1+x_2)^kt^n,\] where $F_n(k)$ is
the number of partitions of $[n]$ in $ \Pi^{(2)}$ such that
$a_2(\pi)=k$.  Remark, we have the relation $F_{n+1}=\sum_{k=1}^{n+1}
F_n(k) k$ by the comment that each partition $\pi$ has $a_2(\pi)$
children.
Under the action of $\mathfrak{S}_2$, the orbit sum has two terms,
one from the identity and one from interchanging $x_1$ and $x_2$:
\begin{multline}
\frac{ ( x_1 + x_2 + 1)  t}
                      {1- t h s}     \sum_{\sigma \in \mathfrak{S}_2}
\epsilon(\sigma) \frac{\sigma\left( x_2^2 \bar{F}(0, x_2;t)\right)}
                     {x_1 x_2^2} \\
\shoveleft{=\Bigl( \sum_{n=0}^{\infty}\left (\frac1{x_1} + \frac1{x_2} +1\right)^n
\left(   ( x_1 + x_2 + 1)  t  \right )^{n+1}
     \Bigr )}
    \times  \Bigl( \frac1{x_1} \bar{F}(0, x_2;t) - \frac{x_1}{x_2^2} \bar{F}(0, x_1;t)
     \Bigr).
\end{multline}
We use the linearity of the operator, and consider this expression in
two steps. First,
We expand this and  apply $\mathcal{CT}$ to the first term, using the definition of $\bar{F}$,  to get
\begin{multline}
\mathcal{CT} \Bigl(
\sum_{n=0}^{\infty}\left (\frac1{x_1} + \frac1{x_2} +1\right)^n\left(   ( x_1 + x_2 + 1)  t  \right )^{n+1}
 \frac1{x_1} \bar{F}(0, x_2;t) \Bigr )\\
=\mathcal{CT}\Bigl( \sum_{n=0}^{\infty}
 \sum_{\substack{0\le l_1, l_2, l_3,\\
                 l_1+l_2 + l_3 = n    }    }
 \sum_{\substack{0 \le j_2, j_3\\
       j_2+j_3 = n-l_1     } }
    \binom{n}{l_1, l_2, l_3} \binom{n+1}{l_1+1 , j_2, j_3}
    x_2^{j_2 - l_2}t^{n+1}  \\
\shoveright{\times
    \sum_{n = 0}^{\infty}
           \sum_{k = 1}^{n+1}  F_n(k)\sum_{i=0}^k\binom{k}{i}
                          x_2^{i} t^n
                 \Bigr)}\\
=\sum_{n=0}^\infty t^n \sum_{n^*\le n-1}
  \sum_{\substack{0\le l_1, l_2, l_3 \\
                 l_1+l_2 + l_3 =  n-n^*-1}}
 \sum_{\substack{0 \le j_2, j_3 \\
       j_2+j_3 =n-n^*-1-l_1  }}
     \binom{n-n^*-1}{l_1, l_2, l_3}
     \binom{n-n^*}{l_1+1, j_2, j_3}\\
\shoveright{\times\sum_{k\le n^*+1} F_{n^*}(k) \binom{k}{l_2-j_2}.}\\
\end{multline}
Remark, to extract the constant coefficient with respect to $x_2$, we impose
$j_2-l_2=-i$ on the inner most summation.

A similar expression is obtained from the second term:
\begin{multline}
\mathcal{CT} \Bigl(
\sum_{n=0}^{\infty}\left(\frac1{x_1} + \frac1{x_2} +1\right)^n
                   \left(   ( x_1 + x_2 + 1)  t  \right)^{n+1}
 \frac{x_1}{x_2} \bar{F}(0, x_1;t) \Bigr )\\
=\mathcal{CT}\Bigl(\sum_{n=0}^{\infty}
 \sum_{\substack{0\le l_1, l_2, l_3,\\
                 l_1+l_2 + l_3 = n    }    }
 \sum_{\substack{0 \le j_1, j_3 \le n,\\
       j_1+j_3 = n-1 -l_2    } }
    \binom{n}{l_1, l_2, l_3} \binom{n+1}{j_1, l_2+2, j_3}
    x_1^{j_1 - l_1+1}t^{n+1}      \\
 \shoveright{\times
    \sum_{n = 0}^{\infty}
    \sum_{k = 1}^{n+1}  F_n(k)\sum_{i=0}^k\binom{k}{i}
                          x_1^{i} t^n
                 \Bigr)}\\
=\sum_{n=0}^\infty t^n \sum_{n^*\le n-1}
  \hspace{-4mm}
  \sum_{\substack{0\le l_1, l_2, l_3 \\
                 l_1+l_2 + l_3 =  n-n^*-1}}
   \sum_{\substack{0 \le j_1, j_3 \\
       j_1+j_3 =n-n^*-2-l_2     } }\hspace{-4mm}
   \binom{n-n^*-1}{l_1, l_2, l_3}
   \binom{n-n^*}{j_1, l_2+2, j_3}\\
\shoveright{\times\sum_{k\le n^*+1} F_{n^*}(k)\binom{k}{l_1-j_1-1}.}\\
\end{multline}
where similar conditions as above also apply to surviving terms,  namely:
$-i = j_1 - l_1 + 1$.

%----------------------------------------------------------------------
\subsection{Third term of RHS of Equation~\eqref{eq:morsum}}
%----------------------------------------------------------------------
Finally, the action of $\mathfrak{S}_2$ on the third term of RHS of
Equation~\eqref{eq:morsum} yields two terms as in the previous case,
and the analysis is almost identical. The third term is:
\begin{align*}
&\frac{ ( x_1 + x_2 + 1)  t}
      {1- t h s}
 \sum_{\sigma \in \mathfrak{S}_2}
   \frac{\epsilon(\sigma)}{x_1 x_2^2}
   \sigma\left( \frac{x_1x_2^2 }
                      {x_2 + 1}\bar{F}(0, x_2;t)\right)\\
&=\left( \sum_{n=0}^{\infty}
            \left(\frac1{x_1} + \frac1{x_2} + 1\right)^n
            \left(   ( x_1 + x_2 + 1)  t   \right)^{n+1}
     \right)
\times \left(
             \frac{\bar{F}(0, x_2;t)}{x_2+1} - \frac{x_1}{x_2}
             \frac{ \bar{F}(0, x_1;t) }{ x_1 + 1}
     \right)
\end{align*}
We take a closer look at the second part:
\begin{equation*}
 \frac{\bar{F}(0, x_2;t)}{x_2+1} - \frac{x_1}{x_2}
             \frac{ \bar{F}(0, x_1;t) }{ x_1 + 1}
=\sum_n\sum_{k=1}^{n+1} F_n(k)t^n
\left(
(x_2 + 1)^{k-1} - \frac{x_1}{x_2}(x_1 + 1)^{k-1}
\right).
\end{equation*}

Applying $\mathcal{CT}$ to the entire expression yields
\begin{multline}
\sum_{n}t^n\sum_{n^*\le n-1}\sum_{k\le n^*+1} F_{n^*}(k)
 \hspace{-5mm}\sum_{\substack{0\le l_1, l_2, l_3 \\
                 l_1+l_2 + l_3 =  n-n^*-1}    } \binom{n-n^*-1}{l_1, l_2, l_3}\\
\left(
   \sum_{\substack{0 \le j_2, j_3 \\
       j_2+j_3 =n-n^*-l_1  } }
     \binom{n-n^*}{l_1, j_2, j_3}
     \binom{k-1}{l_2-j_2}
  -\sum_{\substack{0 \le j_1, j_3 \\
       j_1+j_3 =n-n^*-l_2-1     } }
     \binom{n-n^*}{j_1, l_2+1, j_3}
     \binom{k-1}{l_1-j_1-1}
\right).
\end{multline}

\subsection{A complete expression for $F_n$}
We can put the three components together into one expression for the
coefficient of $F_n$ in terms of $F_{n^*}(k)$, a function of the label where $n^*<n$:
\begin{multline}
F_n =\sum_{\substack{0 \le l_1, l_2, l_3 \\
                                  l_1 + l_2 + l_3 = n }   }
 \binom{n}
       {l_1, l_2, l_3}\left(1-\frac{l_1}{l_2+1}\right)\\
\shoveleft{- \sum_{n^*=0}^{n-1}\sum_{k\le n^*+1} F_{n^*}(k)
 \hspace{-5mm}\sum_{\substack{0\le l_1, l_2, l_3 \\
                 l_1+l_2 + l_3 =  n-n^*-1}    } \binom{n-n^*-1}{l_1,
               l_2, l_3}}\\
\Bigl( \sum_{\substack{0 \le j_2, j_3 \\
       j_2+j_3 =n-n^*-1-l_1  }}
   \binom{n-n^*}{l_1+1, j_2, j_3}
   \binom{k}{l_2-j_2}
- \sum_{\substack{0 \le j_1, j_3\\
       j_1+j_3 = n-l_2 -2     } }
  \binom{n-n^*}{j_1, l_2+2, j_3}
  \binom{k}{l_1-j_1-1}\\
+  \sum_{\substack{0 \le j_2, j_3 \\
       j_2+j_3 =n-n^*-l_1  } }
     \binom{n-n^*}{l_1, j_2, j_3}
     \binom{k-1}{l_2-j_2}
  -\sum_{\substack{0 \le j_1, j_3 \\
       j_1+j_3 =n-n^*-l_2-1     } }
     \binom{n-n^*}{j_1, l_2+1, j_3}
     \binom{k-1}{l_1-j_1-1}
\Bigr).
\end{multline}

Note how $F_n$ is expressed as a convolution-like sum involving all previous $F_{n^*}$ for $n^* < n$. In this form, the authors are unable to obtain a recurrence for the $F_n$'s.

%----------------------------------------------------------------------
\section{Complexity of $m \ge 3$}
\label{sec:complexity}
%----------------------------------------------------------------------
These examples give us a strong flavour of the general formula. The
first term is always the constant term of a rational function, and
hence is always D-finite. There are some sources for added complexity
when $m$ is greater than $2$. First, the number of terms in the orbit
sum grows like $m!$, although one can expect them to  be of a similar
form, as was the case in the $m=2$ case. The number of parameters that
play a role in the formulas is perhaps the key difference. In the
$m=1$ case, we eliminate dependence on the parameter, and determine
direct recurrences. In the $m=2$ case, we use the parameter $a_2$, but
we also have the additional property that the sum of this parameter
over all 3-nonnesting partitions of size $n$ is the number of
3-nonnesting partitions of size $n+1$. Thus, there is an additional
relation. 

We avoid the full treatment of the $m=3$ case, and rather go directly
to the typical effect of $\mathcal{CT}$ on the orbit sum to illustrate how this
expression is increasingly complex. The functional equation in this
case is
\begin{multline}
\frac{x_1 + x_2 + x_3+ 1}
                      {1- t h s}
\sum_{\sigma \in \mathfrak{S}_3}
     \epsilon(\sigma) \frac{\sigma(x_1 x_2^2 x_3^3)}
                           {x_1 x_2^2 x_3^3}
\\
- \frac{x_1 + x_2 + x_3+ 1}
                      {1- t h s}    t
\sum_{\sigma \in \mathfrak{S}_3}
          \frac{\epsilon(\sigma)}{x_1 x_2^2x_3^3}\sigma\left( x_2^2 x_3^3\bar{F}(0,x_2, x_3;t)\right)
\\
 \shoveright{
  - \frac{x_1 + x_2 + x_3+ 1} {1- t h s}  t
   \sum_{\sigma \in \mathfrak{S}_3}
          \frac{\epsilon(\sigma)}{x_1 x_2^2x_3^3}
           \sigma\left(\frac{x_2^2 x_3^3\bar{F}(0,x_2, x_3;t)}
            {x_2+x_3+1}\right).
}
\\
\end{multline}
Applying $\mathcal{CT}$ to the first term yields a sum of six
multinomial summations that simplifies to an expression of the form
\[
\sum_{k=0}^{\infty} \sum_{\substack{0 \le l_1, l_2, l_3, l_4 \le k, \\
                                    l_1 + l_2 + l_3+ l_4 = k }   }
 \binom{k}
       {l_1, l_2, l_3, l_4} R(l_1, l_2, l_3, l_4),
\]
where $R$ is a simple rational function, and, as we noted earlier this
expression is D-finite, since it is a coefficient extraction of a
rational function.

The second and third terms involve
\[\bar{F}(0,x_2, x_3;t)=\sum_{\pi \in \Pi^{(3)}}
(x_2 + x_3+ 1)^{a_2(\pi)} (x_3+1)^{a_3(\pi) - a_2(\pi)} t^{|\pi|}.\]
The result is that when~$\mathcal{L}$ is applied to the second and third terms, we get
the nested summations involving complex expressions of $a_2(\pi)$ and
$a_3(\pi)$. The following is a typical sample
expression for the coefficient of $t^n$:
\begin{multline}
\sum_{k \le n-1} \sum_{\substack{0 \le l_1, l_2, l_3, l_4 \le k, \\
                                                       l_1 + l_2 + l_3 + l_4 = k }   }
 \binom{k}
       {l_1, l_2, l_3, l_4}
            \sum_{\substack{0 \le j_2, j_3, j_4 \le k\\
                                                 l_1+j_2 +j_3+j_4 = k}   }
 \binom{k+1}
              {l_1+1, j_2,j_3, j_4 }
\\
\left(\sum_{\substack{ \pi \in \Pi^{(3)} \\
                                      |\pi| ={ n - k - 1}}}
        \sum_p
                            \binom{a_2(\pi)} {l_2-j_2, p,  a_2(\pi)-p -l_2+j_2}
                 \times \binom{a_3(\pi) - a_2(\pi)}{l_3-p-j_3}
\right).
\end{multline}

%----------------------------------------------------------------------
\section{Computing series expansions}
\label{sec:series}
%----------------------------------------------------------------------
All three Equations~\eqref{eq:u}, \eqref{veq}, and  \eqref{eq:kerx} are used
to generate initial terms in the series. To improve convergence,
we slightly modify the $x$ equation:
\begin{equation}
\begin{split}
\bar{F}(x_1, x_2, \dots, x_m;t) =
s + s t h &\bar{F}(x_1, x_2, \dots, x_m;t) \\
-st\Bigl(&\frac{s}{s-x_1} \frac{\bar{F}(0, x_2, x_3, \dots x_m);t)}{x_1}  \\
&+  \sum_{j=2}^m \frac{\bar{F}(x_1, \dots, x_{j-2}, x_{j-1} + x_j, 0, x_{j+1}, \dots, x_m; t)}{x_j}\Bigr),
\label{eq:kerxis}
\end{split}
\end{equation}

Notice that in Equation~\eqref{eq:kerxis}, if one has a series
expansion of $\bar{F}(\mathbf{x};t)$ correct up to $t^k$, then
substituting this series into RHS of Equation~\eqref{eq:kerxis} yields
the series expansion of $\bar{F}$ correct to $t^{k+1}$ because the RHS
of Equation~\eqref{eq:kerxis} contains a term free of $t$,
otherwise, the degree of $t$ is increased by $1$.  We have thus iterated Equation~\eqref{eq:kerxis}
 to get enumerative data for up to $m=9$.
 
   For $3$-nonnesting set partitions, an average laptop running Maple15 can produce $70$ terms in a reasonable time (less than $24$ hours).
For $m=4$,  only $38$ terms; $m=5$, $27$ terms; $m=6$, $20$ terms; $m=7$, $16$ terms, $m=8$, $12$ terms; and finally $m=9$, $12$ terms. The limitation seems memory space due to the growing complication in the functional equation when $m$ gets larger.
\begin{table}
\centering\tiny
\begin{tabular}{ll|llllllllllllllll}
\multicolumn{17}{c}{$n$}\\
$m$& OEIS \#&
1&2&3&4&5&6&7&8&9&10&11&12&13&14&15  
\\\hline
$1$&A000108& 1& 2& 5& 14& 42& 132& 429& 1430& 4862& 16796& 58786& 208012& 742900& 2674440& 9694845
\\
$2$&A108304& 1& 2& 5& 15& 52& 202& 859& 3930& 19095& 97566& 520257& 2877834& 16434105& 96505490& 580864901
\\
$3$&A108305& 1& 2& 5& 15& 52& 203& 877& 4139& 21119& 115495& 671969& 4132936& 26723063& 180775027& 1274056792
\\
$4$&A192126 &1& 2& 5& 15& 52& 203& 877& 4140& 21147& 115974& 678530& 4212654& 27627153& 190624976& 1378972826
\\
$5$&A192127& 1& 2& 5& 15& 52& 203& 877& 4140& 21147& 115975& 678570& 4213596& 27644383& 190897649& 1382919174
\\
$6$&A192128 & 1& 2& 5& 15& 52& 203& 877& 4140& 21147& 115975& 678570& 4213597& 27644437& 190899321& 1382958475
\\[1mm] \hline

\end{tabular}

\bigskip

\caption{Numbers of set partitions of $n$ avoiding an
  $m+1$-nesting. The OEIS numbers refer to entries in the Online
  Encyclopedia of Integer Sequences~\cite{oeis}}
\label{table:1}
\end{table}

%----------------------------------------------------------------------
\section{Conclusion}
%----------------------------------------------------------------------

Without passing through vacillating lattice walks or  tableaux, the generating tree approach permits a direct translation to a
functional equation involving an arbitrary number of catalytic
variables satisfied by set partitions avoiding $m+1$-nestings for any
$m$. Constant term coefficient extraction analysis gives us insight into why the
number of $3$-nonnesting set partitions should be more easily
controlled than those of higher non-nesting set
partitions. The authors are aware of the techniques developed for constant term extraction and are investigating how such techniques can give insight to the analysis of $m$-nonnesting numbers of set partitions. Though explicit generating
trees are given, formulas thus generated still depend on labels of set partitions. Perhaps further study into the nature of generating trees which
give rise to D-finite series, along the lines of the study
in~\cite{bb-mdfgg-b} will help us understand the differences. 

A second way that might yield a proof of non-D-finiteness would be to
use our expressions to determine bounds on the order and the coefficient degrees of the minimal differential
equation satisfied by the generating function. Though a tantalizingly simple idea, the limitation seems still the lack of data when larger $m$'s give so few values relative to the number one would need to test non-D-finiteness. Nevertheless, this would guide
searches and a fruitless search would then be a definitive result.

Finally, our generating tree approach is limited only to the non-enhanced
case. For a more general treatment of the subject involving enhanced
set partitions and permutations, both enhanced and non-enhanced, we
refer the reader to \cite{bemy} by Burrill, Elizalde, Mishna, and Yen.

%--------------------------------------------------------------------------
\section{Acknowledgements}
We are grateful to an anonymous referee for many constructive suggestions and  to Mireille Bousquet-M\'elou for her suggestions, 
Mogens Lemvig Hansen for his tireless generation of numbers with
Maple. The first author is partially supported by an Natural Sciences and
Engineering Research Council of Canada Discovery Grant.

%----------  Bibliography --------------
%\bibliographystyle{amsplain}
%\bibliography{setpartition}
%%
\providecommand{\bysame}{\leavevmode\hbox to3em{\hrulefill}\thinspace}
\providecommand{\MR}{\relax\ifhmode\unskip\space\fi MR }
\providecommand{\MRhref}[2]{%
  \href{http://www.ams.org/mathscinet-getitem?mr=#1}{#2}
}
\providecommand{\href}[2]{#2}

\end{document}